\documentclass[10pt]{article}

\usepackage{amsmath}
\usepackage{amsfonts}
\usepackage{amsthm}
\usepackage[english]{babel}

\newtheorem{theorem}{Theorem}
\newtheorem{lemma}{Lemma}
\newtheorem{definition}{Definition}
\newtheorem{corollary}{Corollary}

\newtheorem{remark}{Remark}

\newcommand{\mR}{\mathbb{R}}

\newcommand{\mN}{\mathbb{N}}
\newcommand{\mE}{\mathbb{E}}

\newcommand{\mS}{\mathbb{S}}

\newcommand{\cD}{\mathcal{D}}
\newcommand{\cM}{\mathcal{M}}
\newcommand{\cH}{\mathcal{H}}

\newcommand{\cF}{\mathcal{F}}
\newcommand{\cP}{\mathcal{P}}
\newcommand{\cC}{\mathcal{C}}
\newcommand{\cR}{\mathcal{R}}
\newcommand{\cK}{\mathcal{K}}
\newcommand{\cS}{\mathcal{S}}

\newcommand{\ux}{\underline{x}}
\newcommand{\uxb}{\underline{x} \grave{}}

\newcommand{\uyb}{\underline{y} \grave{}}
\newcommand{\uy}{\underline{y}}
\newcommand{\uub}{\underline{u} \grave{}}

\newcommand{\uvb}{\underline{v} \grave{}}

\newcommand{\pj}{\partial_{x_j}}
\newcommand{\pjb}{\partial_{{x \grave{}}_{j}}}

\newcommand{\px}{\partial_x}
\newcommand{\py}{\partial_y}
\newcommand{\upx}{\partial_{\underline{x}}}
\newcommand{\upxb}{\partial_{\underline{{x \grave{}}} }}
\newcommand{\upy}{\partial_{\underline{y}}}
\newcommand{\upyb}{\partial_{\underline{{y \grave{}}} }}

\begin{document}
\title{Fourier transform and related integral transforms in superspace}

\author{H.\ De Bie\thanks{Ph.D. Fellow of the Research Foundation - Flanders (FWO), E-mail: {\tt Hendrik.DeBie@UGent.be}}}

\date{\small{Clifford Research Group -- Department of Mathematical Analysis}\\
\small{Faculty of Engineering -- Ghent University\\ Galglaan 2, 9000 Gent,
Belgium}}

\maketitle

\begin{abstract}
In this paper extensions of the classical Fourier, fractional Fourier and Radon transforms to superspace are studied. Previously, a Fourier transform in superspace was already studied, but with a different kernel. In this work, the fermionic part of the Fourier kernel has a natural symplectic structure, derived using a Clifford analysis approach.
Several basic properties of these three transforms are studied. Using suitable generalizations of the Hermite polynomials to superspace (see [J. Phys. A: Math. Theor. 40 (2007) 10441-10456]) an eigenfunction basis for the Fourier transform is constructed.
\end{abstract}


\section{Introduction}

In a previous set of papers (see \cite{DBS1,DBS3,DBS4,DBS5,DBS6}) we have developed a new approach to the study of superspace. This approach is not based on algebraic geometry as in \cite{MR0580292,MR732126,MR565567}, nor on differential geometry as in \cite{MR574696,MR778559,MR1099318}, but instead on a generalization of Clifford analysis, i.e. a hypercomplex function theory of functions taking values in Clifford algebras (see \cite{MR697564,MR1169463}). In short, this means that by introducing suitable orthogonal and symplectic Clifford algebra generators, we were able to define a.o. a Dirac operator which squares to a super Laplace operator and several other important types of operators in superspace, thus constructing a representation of the Lie superalgebra $\mathfrak{osp}(1|2)$.

The advantage of our approach is that it has allowed a.o. for a strong motivation for the Berezin integral (see \cite{DBS5} and \cite{DBE1}) using notions from classical harmonic analysis. We were also able to construct an elegant generalization of Cauchy's integral formula to superspace (see \cite{DBS7}), thus providing more insight in the nature of fermionic integration.

The aim of the present paper is to use this newly developed framework to study generalizations of the Fourier, the fractional Fourier and the Radon transform to superspace. This is of course not a new idea. Several other authors have already developed a Fourier calculus on superspace. Without claiming completeness, we refer the reader to e.g. \cite{MR1147140,MR778559,MR1192484}.
Apart from the fact that these authors work in different versions of superanalysis, the main difference is that we use another kernel to define the Fourier transform. Without going into further detail now, we will define the fermionic (i.e. with respect to the anti-commuting variables) Fourier transform as:
\[
\cF_{0 | 2n} (f(x))(y) = (2 \pi)^{n} \int_{B,x} e^{i \langle x, y \rangle} \; f(x)
\]
where the integration is the Berezin integral with respect to the $x$ variables and with
\[
\langle x, y \rangle = \langle x, y \rangle_s =  \sum_{j=1}^{n} ({x \grave{}}_{2j-1}{y \grave{}}_{2j} - {x \grave{}}_{2j} {y \grave{}}_{2j-1}).
\]
The other authors, cited above, use the same formula, but with the following kernel
\[
\langle x, y \rangle = \langle x, y \rangle_o =  \sum_{j=1}^{2n} {x \grave{}}_{j}{y \grave{}}_{j}.
\]

The use of this new kernel is quite crucial for the resulting calculus. First of all, we have that $\langle x, y \rangle_s = \langle y, x \rangle_s $  whereas $\langle x, y \rangle_o = -\langle y, x \rangle_o$. Secondly, $\langle x, y \rangle_s$ is invariant under symplectic changes of basis, while $\langle x, y \rangle_o$ is invariant under the orthogonal group.

A drawback of our approach is that we can only consider even numbers of anti-commuting variables, due to the symplectic structure. However, this also has some advantages, in that it makes the Berezin integral an even operator, independently of the dimension. Moreover, in this case the fermionic delta distribution satisfies $\delta(x-y) = \delta(y-x)$ (as classically) which is not the case when considering odd numbers of anti-commuting variables.

Another advantage of our approach is that this Fourier transform behaves nicely with respect to the fermionic Laplace operator (which has the corresponding symplectic structure of the Fourier kernel). This allows us to construct an eigenbasis of the Fourier transform using the Clifford-Hermite functions, introduced in \cite{DBS3}.
As a consequence, we obtain an operator exponential expression of the Fourier transform, which enables us to define a fractional Fourier transform (see e.g. \cite{OZA}) in superspace and to study some of its properties.

We are also able to define a Radon transform in superspace by means of the central-slice theorem (see \cite{MR1723736,MR709591}) which connects the classical Radon transform with two consecutive Fourier transformations. Again we will show that this transform behaves nicely with respect to the Clifford-Hermite functions.

In a subsequent paper (see \cite{DBS8}), we use the here developed Fourier transform to study a Schr\"odinger operator with delta (or Dirac) potential in superspace. Our method yields the wave function and the energy as function of the so-called super-dimension $M$, for all numerical values $M \leq 1$.

The paper is organized as follows. In section \ref{preliminaries} we give a short introduction to Clifford analysis in superspace, focusing on the notions needed in the sequel. In section \ref{BosFermFouriertrafo} we define the fermionic Fourier transform using the symplectic kernel and we study its basic properties. Then in section \ref{gentransform} we define the general Fourier transform, we study its eigenfunctions and we establish its operator exponential form. 
In section \ref{FundSolLaplaceOperator} we apply the general Fourier transform to determine the fundamental solution of the super Laplace operator. 
The operator exponential form of the Fourier transform is used in section \ref{FracFourierTrafo} to define the fractional Fourier transform.
Finally, the previous results on the Fourier transform are used in section \ref{RadonTransform} to define the Radon transform in superspace.

\section{Preliminaries}
\label{preliminaries}

The basic algebra of interest in the study of Clifford analysis in superspace (see \cite{DBS1,DBS4}) is the real algebra $\cP = \mbox{Alg}(x_i, e_i; {x \grave{}}_j,{e \grave{}}_j)$, $i=1,\ldots,m$, $j=1,\ldots,2n$
generated by

\begin{itemize}
\item $m$ commuting variables $x_i$ and $m$ orthogonal Clifford generators $e_i$
\item $2n$ anti-commuting variables ${x \grave{}}_i$ and $2n$ symplectic Clifford generators ${e \grave{}}_i$
\end{itemize}
subject to the multiplication relations
\[ \left \{
\begin{array}{l} 
x_i x_j =  x_j x_i\\
{x \grave{}}_i {x \grave{}}_j =  - {x \grave{}}_j {x \grave{}}_i\\
x_i {x \grave{}}_j =  {x \grave{}}_j x_i\\
\end{array} \right .
\quad \mbox{and} \quad
\left \{ \begin{array}{l}
e_j e_k + e_k e_j = -2 \delta_{jk}\\
{e \grave{}}_{2j} {e \grave{}}_{2k} -{e \grave{}}_{2k} {e \grave{}}_{2j}=0\\
{e \grave{}}_{2j-1} {e \grave{}}_{2k-1} -{e \grave{}}_{2k-1} {e \grave{}}_{2j-1}=0\\
{e \grave{}}_{2j-1} {e \grave{}}_{2k} -{e \grave{}}_{2k} {e \grave{}}_{2j-1}=\delta_{jk}\\
e_j {e \grave{}}_{k} +{e \grave{}}_{k} e_j = 0\\
\end{array} \right .
\]
and where moreover all generators $e_i$, ${e \grave{}}_j$ commute with all variables $x_i$, ${x \grave{}}_j$.

\noindent
Denoting by $\Lambda_{2n}$ the Grassmann algebra generated by the anti-commuting variables ${x \grave{}}_j$ and by $\cC$ the algebra generated by all the Clifford numbers $e_i, {e \grave{}}_j$, we clearly have that
\[
\cP = \mR[x_1,\ldots,x_m]\otimes \Lambda_{2n} \otimes \cC.
\]

In the case where $n = 0$ we have that $\cC \cong \mR_{0,m}$, the standard orthogonal Clifford algebra with signature $(-1,\ldots,-1)$. Similarly, the algebra generated by the ${e \grave{}}_j$ is isomorphic with the Weyl algebra over a vectorspace of dimension $2n$ equipped with the canonical symplectic form.

The most important element of the algebra $\cP$ is the vector variable $x = \ux+\uxb$ with
\[
\begin{array}{lll}
\ux &=& \sum_{i=1}^m x_i e_i\\
&& \vspace{-2mm}\\
\uxb &=& \sum_{j=1}^{2n} {x \grave{}}_{j} {e \grave{}}_{j}.
\end{array}
\]

One easily calculates that
\[
x^2 = \uxb^2 +\ux^2 = \sum_{j=1}^n {x\grave{}}_{2j-1} {x\grave{}}_{2j}  -  \sum_{j=1}^m x_j^2.
\]
If we consider two different vector variables $x$ and $y$, then we define
\[
\langle x,y \rangle = \frac{1}{2} \{ x, y \} =  - \sum_{i=1}^{m} x_i y_i +\frac{1}{2} \sum_{j=1}^{n}({x \grave{}}_{2j-1}{y \grave{}}_{2j} - {x \grave{}}_{2j} {y \grave{}}_{2j-1}).
\]
It is clear that this product is symmetric: $\langle x,y \rangle = \langle y,x \rangle$. Moreover, it is invariant under co-ordinate changes of the type $SO(m)\times Sp(2n)$. In the sequel we will use $\langle x,y \rangle$ as the kernel for our super Fourier transform.

The super Dirac operator is defined as 
\[
\px = \upxb-\upx = 2 \sum_{j=1}^{n} \left( {e \grave{}}_{2j} \partial_{{x\grave{}}_{2j-1}} - {e \grave{}}_{2j-1} \partial_{{x\grave{}}_{2j}}  \right)-\sum_{j=1}^m e_j \pj,
\]
its square being the super Laplace operator
\[
\Delta = \px^2 =4 \sum_{j=1}^n \partial_{{x \grave{}}_{2j-1}} \partial_{{x \grave{}}_{2j}} -\sum_{j=1}^{m} \pj^2.
\]
The bosonic part of the latter operator is $\Delta_b = -\sum_{j=1}^{m} \pj^2$, which is nothing else but the classical Laplace operator; the fermionic part is $\Delta_f = 4 \sum_{j=1}^n \partial_{{x \grave{}}_{2j-1}} \partial_{{x \grave{}}_{2j}}$.

The Euler operator in superspace is defined as
\[
\mE = \sum_{j=1}^m x_j \pj+\sum_{j=1}^{2n} {x \grave{}}_{j} \pjb
\]
and allows for the decomposition of $\cP$ into spaces of homogeneous $\cC$-valued polynomials:
\[
\cP = \bigoplus_{k=0}^{\infty} \cP_k, \quad \cP_k=\left\{ \omega \in \cP \; | \; \mE \omega=k \omega \right\}.
\]

For the other important operators in super Clifford analysis we refer the reader to \cite{DBS1,DBS4}. Letting $\px$ act on $x$ we find that

\[
\px x = x \px = m-2n = M
\]
where $M$ is the so-called super-dimension. This super-dimension is of the utmost importance (see e.g. \cite{DBS5}), as it gives a global characterization of our superspace. The physical meaning of this parameter is discussed in \cite{DBS3}.

The basic calculational rules for the Dirac and Laplace operator on the algebra $\cP$ are given in the following lemma (see \cite{DBS1}). They are a consequence of the operator equality
\[
x \px + \px x = 2 \mE + M.
\]

\begin{lemma}
Let $s \in \mN$ and $R_k \in \cP_k$, then
\begin{eqnarray*}
\px(x^{2s} R_k) &=& 2 s x^{2s-1}R_k + x^{2s} \px R_k\\
\px(x^{2s+1} R_k) &=& (2k + M + 2s) x^{2s}R_k - x^{2s+1} \px R_k\\
\Delta (x^{2s} R_{k})&=& 2s(2k+M+2s-2) x^{2s-2} R_k + x^{2s} \Delta R_k.
\end{eqnarray*}
\end{lemma}

We have the following important definitions:
\begin{definition}
A (super)-spherical harmonic of degree $k$ is an element $H_k \in \cP$ satisfying
\begin{eqnarray*}
\Delta H_k&=&0\\
\mE H_k &=& k H_k, \quad \mbox{i.e. } H_k \in \cP_k.
\end{eqnarray*}
The space of spherical harmonics of degree $k$ will be denoted by $\cH_k$.
\end{definition}

\begin{definition}
A (super)-spherical monogenic of degree $k$ is an element $M_k \in \cP$ satisfying
\begin{eqnarray*}
\Delta M_k&=&0\\
\mE M_k &=& k M_k, \quad \mbox{i.e. } M_k \in \cP_k.
\end{eqnarray*}
The space of spherical monogenics of degree $k$ will be denoted by $\cM_k$.
\end{definition}

The space of spherical harmonics of degree $k$ in the purely bosonic case (i.e. $n=0$) is denoted by $\cH_k^b$. In the purely fermionic case (i.e. $m=0$) we use the notation $\cH_k^f$.

In \cite{DBE1} we have proven the following decomposition of the space of spherical harmonics of degree $k$, which we will need when studying eigenfunctions of the super Fourier transform.
\begin{theorem}[Decomposition of $\cH_k$]
The space $\cH_k$ decomposes under the action of $SO(m)\times Sp(2n)$ into irreducible pieces as follows
\begin{equation}
\cH_{k} = \bigoplus_{i=0}^{\min(n,k)} \cH^b_{k-i} \otimes \cH^f_{i} \;\; \oplus \;\; \bigoplus_{j=0}^{\min(n, k-1)-1} \bigoplus_{l=1}^{\min(n-j,\lfloor \frac{k-j}{2} \rfloor)} f_{l,k-2l-j,j} \cH^b_{k-2l-j} \otimes \cH^f_{j},
\label{decompintoirreps}
\end{equation}
with $f_{l,k-2l-j,j}$ the polynomials given by
\[
f_{k,p,q} = \sum_{i=0}^k \binom{k}{i} \frac{(n-q-i)!}{\Gamma(\frac{m}{2} + p+ k-i)} \ux^{2k-2i} \uxb^{2i}.
\]
\label{completedecomp}
\end{theorem}

For studying integral transforms in superspace we need of course a broader set of functions. For our purposes we define the function spaces
\[
\cF(\Omega)_{m|2n} = \cF(\Omega) \otimes \Lambda_{2n} \otimes \cC
\]
where $\Omega$ is an open set in $\mR^m$ and where $\cF(\Omega)$ stands for $\cD(\Omega)$, $C^{k}(\Omega)$, $L_{p}(\Omega)$, $L_{1}^{\mbox{\footnotesize loc}}(\Omega)$, $\cS(\mR^m)$, $\ldots$.

Integration over superspace is given by the so-called Berezin integral (see \cite{MR0208930,MR732126}), defined as
\[
\int_{\mR^{m|2n},x} f = \int_{\mR^m} dV(\ux) \int_{B,x} f
\]
with
\[
\int_{B,x} = \pi^{-n} \partial_{{x \grave{}}_{2n}} \ldots \partial_{{x \grave{}}_{1}} = \frac{(-1)^n \pi^{-n}}{4^n n!} \upxb^{2n}.
\]
The subscript $x$ means that we are integrating with respect to the $x$ variables. The numerical factor $\pi^{-n}$ is necessary to obtain more symmetric formulae. It also appears naturally in our treatment of the Berezin integral (see \cite{DBS5}, theorem 11).

This definition means that one first has to differentiate $f$ with respect to all anti-commuting variables and then to integrate w.r.t. the commuting variables in the usual way. This integration recipe may seem rather haphazard but it can be explained in a satisfactory way using harmonic analysis in superspace (see \cite{DBS5,DBE1}).

We can also introduce a super Dirac distribution as
\[
\delta(x) = \pi^{n} \delta(\ux) {x \grave{}}_{1} \ldots {x \grave{}}_{2n} = \frac{\pi^{n}}{n!} \delta(\ux) \uxb^{2n}
\]
with $\delta(\ux)$ the classical Dirac distribution in $\mR^m$. As we always work with an even number of anti-commuting variables, we have that both the Berezin integral and the Dirac distribution are even objects, independently of the dimension.

Let us now introduce the Clifford-Hermite polynomials (see \cite{DBS3}) by means of their Rodrigues formula. Let $M_k$ be a spherical monogenic of degree $k$, then we define
\[
CH_{t,M}(M_k)(x)= \exp(-x^2/2) (\px + x)^t  \exp(x^2/2) M_k = \exp(-x^2) (\px )^t  \exp(x^2) M_k.
\]
It is easily seen that
\[
CH_{t,M}(M_k)(x)= CH_{t,M,k}(x) M_k
\]
with $CH_{t,M,k}(x)$ a polynomial in the vector variable $x$, not depending on the precise form of $M_k$, but only on the integer $k$.

These Clifford-Hermite polynomials have an important physical application: they are eigenfunctions of a harmonic oscillator in superspace (see \cite{DBS3}). In the sequel we will also need a rescaled version of them, defined by
\[
\widetilde{CH}_{t,M}(M_k)(x)= \exp(-x^2/2) (\px )^t  \exp(x^2/2) M_k.
\]

We can also construct scalar-valued versions of the Clifford-Hermite polynomials. This is due to the fact that $(\px + x)^2$ is a scalar operator:
\[
(\px + x)^2 = \Delta + x^2 + 2 \mE + M.
\]
We hence obtain the scalar valued polynomials
\[
CH_{2t,M}(H_k)(x)= \exp(-x^2/2) (\px + x)^{2t}  \exp(x^2/2) H_k
\]
with $H_k$ a spherical harmonic of degree $k$.

Note that in the case where $M \not \in -2\mN$ the Clifford-Hermite polynomials form a basis for $\cP$. Consequently, the ($\cC$-valued) functions $\phi_{j,k,l}(x)$ defined by
\begin{eqnarray*}
\phi_{j,k,l}(x) &=& (\px + x )^j M_k^{(l)} \exp{x^2/2}\\
&=& CH_{j,M,k}(x) M_k^{(l)} \exp{x^2/2},
\end{eqnarray*}
with $M_k^{(l)}$ a basis of $\cM_k$, indexed by $l$, form a basis for the function space $\cS(\mR^m)_{m|2n}$. Similarly, we have that the (scalar) functions $\psi_{j,k,l}(x)$ defined by
\begin{eqnarray*}
\psi_{j,k,l}(x) &=& (\px + x )^{2j} H_k^{(l)} \exp{x^2/2}\\
&=& CH_{2j,M,k}(x) H_k^{(l)} \exp{x^2/2},
\end{eqnarray*}
with $H_k^{(l)}$ a basis of $\cH_k$, indexed by $l$, form a basis for the space $\cS(\mR^m)\otimes \Lambda_{2n}$.

Finally, we define the convolution of two functions $f$ and $g$ as
\[
f *g(u) = \int_{\mR^{m|2n},x} f(u-x)g(x).
\]
This operation is clearly not commutative; however note that
\begin{eqnarray*}
f * g(u) &=&\int_{\mR^{m|2n},x} f(u-x)g(x)\\
&=& \int_{\mR^{m|2n},x} \int_{\mR^{m|2n},y} \delta(y+x-u) f(y) g(x)\\
&=& \int_{\mR^{m|2n},y} f(y) \int_{\mR^{m|2n},x} \delta(y+x-u)  g(x)\\
&=& \int_{\mR^{m|2n},y} f(y)g(u-y).
\end{eqnarray*}

\section{The bosonic and fermionic Fourier transform}
\label{BosFermFouriertrafo}

\subsection{The bosonic Fourier transform}
We define the classical Fourier transform of a function $f \in L_1(\mR^m)$ as follows:
\[
\cF_{m | 0}^{\pm}(f)(y) =(2 \pi)^{-m/2} \int_{\mR^m} d V(\ux) \exp{\left(\pm i \sum_{i=1}^{m} x_i y_i   \right)} f(x).
\]
The properties of this transform are very well known, we refer the reader to e.g. \cite{MR0304972}.
In the sequel we will need the following important theorem.

\begin{theorem}
Let $H_l(x) \in \cH_l^b$ be a spherical harmonic of degree $l$. Then one has
\[
\cF_{m | 0}^{\pm} (H_l(x) \exp(\ux^2/2))(y) = (\pm i)^l H_l(y) \exp(\uy^2/2).
\]
\label{bosharmfour}
\end{theorem}
\begin{proof}
See e.g. \cite{MR0304972}.
\end{proof}

\subsection{The fermionic Fourier transform}

We start by introducing the following kernel:
\begin{equation}
K^{\pm}(x,y) = \exp( \mp \frac{i}{2} \sum_{j=1}^{n} ({x \grave{}}_{2j-1}{y \grave{}}_{2j} - {x \grave{}}_{2j} {y \grave{}}_{2j-1})).
\label{fermkernel}
\end{equation}

Easy calculations show that
\[
\begin{array}{lll}
\partial_{{x\grave{}}_{2i}} K^{\pm}(x,y) & = & \pm \frac{i}{2} {y \grave{}}_{2i-1} K^{\pm}(x,y)\\
\vspace{-3mm}\\
\partial_{{x\grave{}}_{2i-1}} K^{\pm}(x,y) & = & \mp \frac{i}{2} {y \grave{}}_{2i} K^{\pm}(x,y)\\
\vspace{-3mm}\\
\partial_{{y\grave{}}_{2i}} K^{\pm}(x,y) & = & \pm \frac{i}{2} {x \grave{}}_{2i-1} K^{\pm}(x,y)\\
\vspace{-3mm}\\
\partial_{{y\grave{}}_{2i-1}} K^{\pm}(x,y) & = & \mp \frac{i}{2} {x \grave{}}_{2i} K^{\pm}(x,y).
\end{array}
\]

Using the kernel (\ref{fermkernel}) we define the fermionic Fourier transform on the Grassmann algebra $\Lambda_{2n}$, generated by the ${x \grave{}}_i$, as
\begin{equation}
\cF_{0 | 2n}^{\pm} (\cdot)(y) = (2\pi)^n \int_{B,x} K^{\pm}(x,y) \; (\cdot)
\end{equation}

\begin{remark}
Note that the kernel $K^{\pm}(x,y)$ is symmetric: 
\[
K^{\pm}(x,y) = K^{\pm}(y,x).
\]
This is clearly not the case if one would use the kernel
\[
\exp( \mp i \sum_{j=1}^{2n} {x \grave{}}_{j}{y \grave{}}_{j})
\]
as is done in e.g. \cite{MR1147140,MR778559,MR1192484}.
Moreover, our definition is invariant under symplectic changes of variables, whereas the other approach is invariant under the orthogonal group.
\end{remark}

Now we have the following basic lemma.
\begin{lemma}
If $g \in \Lambda_{2n}$, then one has:
\[
\begin{array}{lll}
\cF_{0 | 2n}^{\pm} (\partial_{{x\grave{}}_{2i}} g) & = & \mp \frac{i}{2} {y \grave{}}_{2i-1} \cF_{0 | 2n}^{\pm} ( g)\\
\vspace{-3mm}\\
\cF_{0 | 2n}^{\pm} (\partial_{{x\grave{}}_{2i-1}} g) & = & \pm \frac{i}{2} {y \grave{}}_{2i} \cF_{0 | 2n}^{\pm} ( g)\\
\vspace{-3mm}\\
\cF_{0 | 2n}^{\pm} ({x\grave{}}_{2i} g) & = & \pm 2i \; \partial_{{y\grave{}}_{2i-1}} \cF_{0 | 2n}^{\pm} ( g)\\
\vspace{-3mm}\\
\cF_{0 | 2n}^{\pm} ({x\grave{}}_{2i-1} g) & = & \mp 2i \; \partial_{{y\grave{}}_{2i}} \cF_{0 | 2n}^{\pm} ( g).
\end{array}
\]
\label{fermcalculusrules}
\end{lemma}
\begin{proof}
We only prove the first relation, the other proofs being completely similar:
\begin{eqnarray*}
\cF_{0 | 2n}^{\pm} (\partial_{{x\grave{}}_{2i}} g) &=& (2\pi)^n \int_{B,x} K^{\pm}(x,y)  (\partial_{{x\grave{}}_{2i}} g)\\
&=& (2\pi)^n \int_{B,x} \partial_{{x\grave{}}_{2i}} \left[ K^{\pm}(x,y)   g \right] - (2\pi)^n \int_{B,x} \left[ \partial_{{x\grave{}}_{2i}} K^{\pm}(x,y) \right]  g\\
&=& - (2\pi)^n \int_{B,x} \left[ \partial_{{x\grave{}}_{2i}} K^{\pm}(x,y) \right]  g\\
&=& \mp \frac{i}{2} {y\grave{}}_{2i-1} (2\pi)^n \int_{B,x}   K^{\pm}(x,y)  g\\
&=&\mp \frac{i}{2} {y \grave{}}_{2i-1} \cF_{0 | 2n}^{\pm} ( g).
\end{eqnarray*}
\end{proof}

Next we consider the action of the Dirac operator and the vector variable. Using the previous lemma this immediately leads to the following.
\begin{corollary}
One has the following relations
\[
\begin{array}{lllllll}
\cF_{0 | 2n}^{\pm} (\upxb  g) & = & \pm i \uyb  \cF_{0 | 2n}^{\pm} ( g)&\;&\cF_{0 | 2n}^{\pm} (\Delta_f g) & = & - \uyb^2 \cF_{0 | 2n}^{\pm} ( g)\\
\vspace{-3mm}\\
\cF_{0 | 2n}^{\pm} (\uxb g) & = & \pm i \upyb \cF_{0 | 2n}^{\pm} ( g)&\;&\cF_{0 | 2n}^{\pm} (\uxb^2 g) & = & - \Delta_f \cF_{0 | 2n}^{\pm} ( g).
\end{array}
\]
\end{corollary}

Now let us calculate the fermionic Fourier transform of some simple functions:

(i) $\cF_{0 | 2n}^{\pm} (\uxb^{2n}) = n! \cF_{0 | 2n}^{\pm} ({x \grave{}}_{1} \ldots {x \grave{}}_{2n}) = n! 2^n$

(ii) $\cF_{0 | 2n}^{\pm} (\uxb^{2k})$

We first need the following formula
\begin{eqnarray*}
\upxb^{2n-2k} \uxb^{2n} &=& 2n (2n-2-2n) \upxb^{2n-2k-2} \uxb^{2n-2}\\
&=& \ldots\\
&=& 2^{2n-2k} (-1)^{n-k} \frac{n!(n-k!)}{k !} \uxb^{2k}.
\end{eqnarray*}
We then obtain the following result
\begin{eqnarray*}
\cF_{0 | 2n}^{\pm} (\uxb^{2k}) &=& \frac{(-1)^{n-k}}{2^{2n-2k}} \frac{k!}{n! (n-k)!} \cF_{0 | 2n}^{\pm} (\upxb^{2n-2k} \uxb^{2n})\\
&=&  \frac{1}{2^{2n-2k}} \frac{k!}{n! (n-k)!}  \uyb^{2n-2k}\cF_{0 | 2n}^{\pm} ( \uxb^{2n})\\
&=&  2^{2k-n} \frac{k!}{(n-k)!}  \uyb^{2n-2k}.
\end{eqnarray*}

(iii) the fermionic Fourier transform of the Gaussian $\exp(\uxb^2/2)$:
\begin{eqnarray*}
\cF_{0 | 2n}^{\pm} (\exp(\uxb^2/2)) &=& \sum_{j=0}^{n} \cF_{0 | 2n}^{\pm} (\frac{\uxb^{2j}}{2^j j!})\\
&=&\sum_{j=0}^{n} \frac{2^{2j-n}}{2^j j!} \frac{j!}{(n-j)!} \uyb^{2n-2j}\\
&=& \sum_{j=0}^{n} \frac{\uyb^{2n-2j}}{2^{n-j}(n-j)!} \\
&=& \exp(\uyb^2/2).
\end{eqnarray*}
So we conclude that the Gaussian function is invariant under the fermionic Fourier transform (as would be expected).

Now we turn our attention to the inversion of the Fourier transform. We have the following theorem.
\begin{theorem}[inversion]
One has that
\[
\cF_{0 | 2n}^{\pm} \circ \cF_{0 | 2n}^{\mp} = \mbox{id}_{\Lambda_{2n}}.
\]
\end{theorem}
\begin{proof}
As the Fourier transform is linear, it suffices to give the proof for a monomial $x_A = {x \grave{}}_{1}^{\alpha_1} \ldots {x \grave{}}_{2n}^{\alpha_{2n}}$ with $\alpha_i \in \left\{ 0, 1\right\}$. We then find, using lemma \ref{fermcalculusrules}:
\begin{eqnarray*}
\cF_{0 | 2n}^{+} \cF_{0 | 2n}^{-} (x_A) &=& (2i)^{|A|} (-1)^{\sum \alpha_{2i}}\cF_{0 | 2n}^{+} \left( \partial_{{y \grave{}}_{2}}^{\alpha_{1}} \partial_{{y \grave{}}_{1}}^{\alpha_{2}} \ldots \partial_{{y \grave{}}_{2n}}^{\alpha_{2n-1}} \partial_{{y \grave{}}_{2n-1}}^{\alpha_{2n}}  \cF_{0 | 2n}^{-} (1)\right)\\
&=& (2i)^{|A|} (-1)^{\sum \alpha_{2i}} (\frac{i}{2})^{|A|}  (-1)^{\sum \alpha_{2i-1}} {x \grave{}}_{1}^{\alpha_1} \ldots {x \grave{}}_{2n}^{\alpha_{2n}}\cF_{0 | 2n}^{+} \cF_{0 | 2n}^{-} (1)\\
&=& x_A \cF_{0 | 2n}^{+} (\frac{2^{-n}}{n!} \uyb^{2n})\\
&=& x_A.
\end{eqnarray*}
Similarly we find that $\cF_{0 | 2n}^{-} \cF_{0 | 2n}^{+} (x_A) = x_A$.
\end{proof}

As a consequence we immediately obtain that $\cF_{0 | 2n}^{\pm} $ is an isomorphism of $\Lambda_{2n}$.

Now we take the first important step in the construction of an eigenbasis of the Fourier transform. 
\begin{theorem}
Let $H_l(x) \in \cH_l^f$ be a spherical harmonic of degree $l$. Then one has
\[
\cF_{0 | 2n}^{\pm} (H_l(x) \exp(\uxb^2/2))(y) = (\pm i)^l H_l(y) \exp(\uyb^2/2).
\]
\label{fermharmfour}
\end{theorem}
\begin{proof}
As we have that (see \cite{DBS5}, theorem 10)
\[
\int_{B,x} \exp(\uxb^2/2) R = \sum_{k=0}^{n} \frac{(-1)^k (2 \pi)^{-n}}{2^k k!} (\Delta_f^k R)(0)
\]
with $R \in \Lambda_{2n}$, we find that
\begin{eqnarray*}
\cF_{0 | 2n}^{\pm} (H_l(x) \exp(\uxb^2/2)) &=& \sum_{k=0}^{n} \frac{(-1)^k }{2^k k!} \Delta_f^k \left(K^{\pm}(x,y) H_l(x) \right)(0)\\
&=&\sum_{k=0}^{n} \sum_{j=0}^{2n} \frac{(-1)^k (\pm i)^j}{2^k k! j!} \Delta_f^k \left( (-1)^j \langle \uxb,\uyb \rangle^j  H_l(x) \right)(0)\\
&=&\sum_{k \geq l/2}^{n}  \frac{(-1)^k (\pm i)^{2k-l}}{2^k k! (2k-l)!} \Delta_f^k \left( (-1)^{2k-l} \langle \uxb,\uyb \rangle^{2k-l}  H_l(x) \right) \\
\end{eqnarray*}
with $\langle \uxb,\uyb \rangle = \sum_{j=1}^{n} ({x \grave{}}_{2j-1}{y \grave{}}_{2j} - {x \grave{}}_{2j} {y \grave{}}_{2j-1})/2$.

In a similar way as in \cite{DBS5}, theorem 7 (Funk-Hecke), it is proven that the following relation holds:
\[
\Delta^s_f (-1)^k \langle \uxb,\uyb \rangle^k H_l(x) =  \frac{2^{2s} s! n !}{ (n-s)!} \alpha_l^*(t^k) H_l(y) \uyb^{k-l} (-1)^{(k-l)/2}
\]
with $k+l = 2s \leq 2n$ and with
\begin{eqnarray*}
\alpha_l^*(t^k) &=& \frac{k! (n-\frac{k+l}{2})!}{(k-l)! n!} \frac{ \pi^{-\frac{1}{2}} (-1)^{\frac{k+l}{2}}}{2^l} \Gamma(\frac{k-l+1}{2})
  \quad \mbox{if $k+l$ even, $k \geq l$}\\
&=&0 \quad \mbox{if $k+l$ odd}\\
&=&0 \quad \mbox{if $k<l$}.\\
\end{eqnarray*}
Now
\begin{eqnarray*}
&&\sum_{k \geq l/2}^{n}  \frac{(-1)^k (\pm i)^{2k-l}}{2^k k! (2k-l)!} \Delta_f^k \left( (-1)^{2k-l} \langle \uxb,\uyb \rangle^{2k-l}  H_l(x) \right)\\
&=& \sum_{p=0}^{n-l}  \frac{(-1)^{l+p} (\pm i)^{l+2p}}{2^{l+p} (l+p)! (l+2p)!} \Delta_f^{l+p} \left( (-1)^{l+2p} \langle \uxb,\uyb \rangle^{l+2p}  H_l(x) \right)\\
&=& \sum_{p=0}^{n-l}  \frac{(-1)^{l+p} (\pm i)^{l+2p}}{2^{l+p} (l+p)! (l+2p)!} \frac{2^{2l+2p} (l+p)! n !}{ (n-l-p)!}  \alpha_l^*(t^{l+2p})  H_l(y) \uyb^{2p} (-1)^{p}\\
&=& \sum_{p=0}^{n-l}  \frac{(-1)^{l+p} (\pm i)^{l+2p}}{  (l+2p)!} \frac{2^{l+p}  n !}{ (n-l-p)!} \frac{(l+2p)! (n-l-p)!}{(2p)! n!} \frac{ \pi^{-\frac{1}{2}} (-1)^{l+p}}{2^l} \Gamma(\frac{2p+1}{2})  H_l(y) \uyb^{2p} (-1)^{p}\\
&=& (\pm i)^{l}\sum_{p=0}^{n-l} \frac{\uyb^{2p}}{2^p p!}   H_l(y)  \\
&=& (\pm i)^{l}\sum_{p=0}^{n} \frac{\uyb^{2p}}{2^p p!}   H_l(y)  \\
&=&(\pm i)^{l} H_l(y) \exp(\uyb^2/2)
\end{eqnarray*}
where we have used the fact that $\uyb^{2k} H_l(y) =0$ if $k>n-l$.
\end{proof}

As a consequence, we obtain the following theorem, which completely characterizes the fermionic Fourier transform because of the following decomposition (see \cite{DBE1})
\begin{equation}
\Lambda_{2n} = \bigoplus_{k=0}^{n} \left(\bigoplus_{j=0}^{n-k} \uxb^{2j} \cH_k^f \right).
\label{Fischer}
\end{equation}

\begin{theorem}
One has that
\[
\cF_{0 | 2n}^{\pm} (\uxb^{2k} \cH_l^f) = (\pm i)^l \frac{2^{2k+l-n} k!}{(n-k-l)!} \uyb^{2n-2k-2l} \cH_l^f
\]
with $l \leq n$ and $k \leq n-l$.
\end{theorem}

\begin{proof}
Immediately by using theorem \ref{fermharmfour} and noting that $\cF_{0 | 2n}^{\pm}$ maps $k$-homogeneous elements of $\Lambda_{2n}$ to $(2n-k)$-homogeneous elements.
\end{proof}

We also have a Parseval theorem for the fermionic Fourier transform.
\begin{theorem}[Parseval]
Let $f, g \in \Lambda_{2n}$. Then one has:
\[
\int_{B,x} f(x) \overline{g(x)} = \int_{B,y} \cF_{0 | 2n}^{\pm} (f)(y) \overline{\cF_{0 | 2n}^{\pm} (g)(y)}
\]
where the bar denotes the standard complex conjugation.
\label{fermparseval}
\end{theorem}

\begin{proof}
The right-hand side is calculated as follows:
\begin{eqnarray*}
\int_{B,y} \cF_{0 | 2n}^{\pm} (f)(y) \overline{\cF_{0 | 2n}^{\pm} (g)(y)} &=& (2 \pi)^{2n} \int_{B,y} \int_{B,u} \int_{B,v} e^{\mp i \langle \uyb, \uub \rangle} e^{\pm i \langle \uyb, \uvb \rangle} f(u) \overline{g(v)}\\
&=& (2 \pi)^{n} \int_{B,u} \int_{B,v} \left[(2 \pi)^{n}  \int_{B,y}  e^{\mp i \langle \uyb, \uub -\uvb\rangle} \right] f(u) \overline{g(v)}\\
&=& (2 \pi)^{n} \int_{B,u} \int_{B,v} \frac{2^{-n}}{n!} (\uub-\uvb )^{2n} f(u) \overline{g(v)}\\
&=& \int_{B,u} \int_{B,v} \delta(\uub-\uvb) f(u) \overline{g(v)}\\
&=&\int_{B,u} f(u) \overline{g(u)}.
\end{eqnarray*}
\end{proof}

\section{General Fourier transform}
\label{gentransform}

Now we define a Fourier transform on the whole superspace by
\begin{eqnarray*}
\cF_{m | 2n}^{\pm} &=& (2 \pi)^{-M/2} \int_{\mR^m} d V(\ux) \int_{B,x} \exp{(\mp i \langle x , y \rangle)}\\
&=&(2 \pi)^{-M/2} \int_{\mR^m} d V(\ux) \int_{B,x} \exp{\left(\pm i \left(\sum_{i=1}^{m} x_i y_i - \frac{1}{2} \sum_{j=1}^{n}({x \grave{}}_{2j-1}{y \grave{}}_{2j} - {x \grave{}}_{2j} {y \grave{}}_{2j-1}) \right) \right)}.
\end{eqnarray*}

We immediately have that
\begin{equation}
\cF_{m | 2n}^{\pm} = \cF_{m | 0}^{\pm} \circ \cF_{0 | 2n}^{\pm} = \cF_{0 | 2n}^{\pm} \circ \cF_{m | 0}^{\pm} 
\label{splitting}
\end{equation}
which allows us to use the results of the previous section to study the general Fourier transform.

\begin{theorem}[Inversion]
One has that
\[
\cF_{m | 2n}^{\pm} \circ \cF_{m | 2n}^{\mp} = \mbox{id}_{\cS(\mR^m)_{m|2n}}
\]
and consequently that $\cF_{m | 2n}^{\pm}$ is an isomorphism of $\cS(\mR^m)_{m|2n}$.
\end{theorem}
\begin{proof}
We have the following calculation
\begin{eqnarray*}
\cF_{m | 2n}^{+} \circ \cF_{m | 2n}^{-}  &=& \cF_{m | 0}^{+} \circ \cF_{0 | 2n}^{+} \circ \cF_{0 | 2n}^{-} \circ \cF_{m | 0}^{-}\\
&=& \cF_{m | 0}^{+} \circ \mbox{id}_{\Lambda_{2n}} \circ \cF_{m | 0}^{-}\\
&=& \mbox{id}_{\cS(\mR^m)_{m|2n}}.
\end{eqnarray*}
\end{proof}

The action of derivatives and variables is summarized in the following lemma. 
\begin{lemma}
If $g \in \cS(\mR^m)_{m|2n}$, then the following relations hold:
\[
\begin{array}{lllllll}
\cF_{m| 2n}^{\pm} (\partial_{x_i} g) & = & \mp i y_i \cF_{m | 2n}^{\pm} ( g)& \;&\cF_{m| 2n}^{\pm} (x_i g) & = & \mp i \partial_{y_i} \cF_{m | 2n}^{\pm} ( g)\\
\vspace{-3mm}\\
\cF_{m| 2n}^{\pm} (\partial_{{x\grave{}}_{2i}} g) & = & \mp \frac{i}{2} {y \grave{}}_{2i-1} \cF_{m | 2n}^{\pm} ( g)
&\;&\cF_{m | 2n}^{\pm} ({x\grave{}}_{2i} g) & = & \pm 2i \; \partial_{{y\grave{}}_{2i-1}} \cF_{m | 2n}^{\pm} ( g)\\
\vspace{-3mm}\\
\cF_{m | 2n}^{\pm} (\partial_{{x\grave{}}_{2i-1}} g) & = & \pm \frac{i}{2} {y \grave{}}_{2i} \cF_{m | 2n}^{\pm} ( g)&\;& \cF_{m| 2n}^{\pm} ({x\grave{}}_{2i-1} g) & = & \mp 2i \; \partial_{{y\grave{}}_{2i}} \cF_{m | 2n}^{\pm} ( g).
\end{array}
\]
The action of the Dirac operator and the vector variable is given by
\[
\begin{array}{lllllll}
\cF_{m | 2n}^{\pm} (\px  g) & = & \pm i y  \cF_{m | 2n}^{\pm} ( g)&\;& \cF_{m | 2n}^{\pm} (\Delta  g) & = & - y^2  \cF_{m | 2n}^{\pm} ( g)\\
\vspace{-3mm}\\
\cF_{m | 2n}^{\pm} (x g) & = & \pm i \py \cF_{m | 2n}^{\pm} ( g)&\;&\cF_{m | 2n}^{\pm} (x^2 g) & = & - \py^2 \cF_{m | 2n}^{\pm} ( g).
\vspace{-3mm}
\end{array}
\]
\label{superfouriercalculus}
\end{lemma}
\begin{proof}
Immediate, using formula (\ref{splitting}), lemma \ref{fermcalculusrules} and the basic properties of the bosonic Fourier transform.
\end{proof}

We also have the following lemma concerning the Fourier transform of the convolution of two functions.
\begin{lemma}
Let $f, g$ be elements of $L_1(\mR^m)_{m|2n}$. Then the following holds
\[
\cF_{m | 2n}^{\pm} (f*g) = (2 \pi )^{M/2} \cF_{m | 2n}^{\pm} (f) \cF_{m | 2n}^{\pm} (g).
\]
\end{lemma}

\begin{proof}
We calculate the left-hand side as
\begin{eqnarray*}
(2 \pi)^{M/2} \cF_{m | 2n}^{\pm} (f*g) &=&  \int_{\mR^{m|2n},u} \int_{\mR^{m|2n},x} \exp{(\mp i \langle u , y \rangle)} f(u-x) g(x)\\
&=& \int_{\mR^{m|2n},x} \left[ \int_{\mR^{m|2n},u} \exp{(\mp i \langle u-x , y \rangle)} f(u-x) \right] \exp{(\mp i \langle x , y \rangle)} g(x)\\
&=&\int_{\mR^{m|2n},x} (2 \pi)^{M/2}  \cF_{m | 2n}^{\pm}(f)(y) \exp{(\mp i \langle x , y \rangle)} g(x)\\
&=& (2 \pi)^{M}\cF_{m | 2n}^{\pm}(f)(y) \cF_{m | 2n}^{\pm}(g)(y) 
\end{eqnarray*}
which completes the proof.
\end{proof}

Now we have the following full Parseval theorem.
\begin{theorem}[Parseval]
Let $f, g \in L_2(\mR^m)_{m|2n}$. Then the following holds:
\[
\int_{\mR^{m|2n},x} f(x) \overline{g(x)} = \int_{\mR^{m|2n},y} \cF_{m | 2n}^{\pm} (f)(y) \overline{\cF_{m | 2n}^{\pm} (g)(y)}.
\]
\end{theorem}

\begin{proof}
This follows immediately from the classical Parseval theorem combined with theorem \ref{fermparseval}.
\end{proof}

Let us now calculate the Fourier transform of the super Gaussian function:
\begin{eqnarray*}
\cF_{m | 2n}^{\pm} (\exp(x^2/2)) &=& \cF_{m | 2n}^{\pm} (\exp(\ux^2/2)\exp(\uxb^2/2))\\
&=& \cF_{m | 0}^{\pm} (\exp(\ux^2/2)) \cF_{0 | 2n}^{\pm} \exp(\uxb^2/2))\\
&=& \exp(\uy^2/2)\exp(\uyb^2/2)\\
&=& \exp(y^2/2).
\end{eqnarray*}

The theorems \ref{bosharmfour} and \ref{fermharmfour} can be merged into the following theorem. The proof makes extensive use of the Clifford-Hermite polynomials introduced in section \ref{preliminaries}.
\begin{theorem}
Let $H_l(x) \in \cH_l$ be a spherical harmonic of degree $l$. Then one has
\[
\cF_{m | 2n}^{\pm} (H_l(x) \exp(x^2/2))(y) = (\pm i)^l H_l(y) \exp(y^2/2).
\]
\label{fouriertransformsphericalharm}
\end{theorem}

\begin{proof}
Due to theorem \ref{completedecomp} it suffices to give the proof in the case where $H_l(x)$ is of the following form:
\[
H_l(x) = f_{k,l-2k-j,j} H^b_{l-2k-j} H^f_{j}
\]
with
\[
f_{k,l-2k-j,j}(\ux^2,\uxb^2) =  \sum_{i=0}^k \binom{k}{i} \frac{(n-j-i)!}{\Gamma(\frac{m}{2} + l-k-j-i)} \ux^{2k-2i} \uxb^{2i}
\]
and with $H^b_{l-2k-j} \in \cH_{l-2k-j}^b$, $H^f_{j} \in \cH_j^f$.

We then calculate the Fourier transform of $H_l(x) \exp(x^2/2)$ as
\begin{eqnarray*}
&&\cF_{m | 2n}^{\pm} (H_l(x) \exp(x^2/2))\\ 
&=& \cF_{m | 2n}^{\pm} ( f_{k,l-2k-j,j} H^b_{l-2k-j} H^f_{j} \exp(x^2/2))\\
&=& (-1)^k \sum_{i=0}^k \binom{k}{i} \frac{(n-j-i)!}{\Gamma(\frac{m}{2} + l-k-j-i)} \upy^{2k-2i} \upyb^{2i} \cF_{m | 2n}^{\pm} (  H^b_{l-2k-j} H^f_{j} \exp(x^2/2))\\
&=& (-1)^k \sum_{i=0}^k \binom{k}{i} \frac{(n-j-i)!}{\Gamma(\frac{m}{2} + l-k-j-i)} \upy^{2k-2i} \upyb^{2i} \cF_{m | 0}^{\pm} (  H^b_{l-2k-j} \exp(\ux^2/2))  \cF_{0 | 2n}^{\pm} ( H^f_{j} \exp(\uxb^2/2))\\
&=& (\pm i)^l \sum_{i=0}^k \binom{k}{i} \frac{(n-j-i)!}{\Gamma(\frac{m}{2} + l-k-j-i)} \upy^{2k-2i} \upyb^{2i}  H^b_{l-2k-j}(\uy) \exp(\uy^2/2)   H^f_{j}(\uyb) \exp(\uyb^2/2)\\
&=& (\pm i)^l \sum_{i=0}^k \binom{k}{i} \frac{(n-j-i)!}{\Gamma(\frac{m}{2} + l-k-j-i)} \widetilde{CH}_{ 2k-2i,m,l-2k-j}(\uy) \widetilde{CH}_{ 2i,-2n,j}(\uyb)\\
&& \times  H^b_{l-2k-j}(\uy) \exp(\uy^2/2)   H^f_{j}(\uyb) \exp(\uyb^2/2).
\end{eqnarray*}

So we still need to prove that
\begin{equation}
\sum_{i=0}^k \binom{k}{i} \frac{(n-j-i)!}{\Gamma(\frac{m}{2} + l-k-j-i)} \widetilde{CH}_{ 2k-2i,m,l-2k-j}(\uy) \widetilde{CH}_{ 2i,-2n,j}(\uyb) = f_{k,l-2k-j,j}(\uy^2,\uyb^2).
\label{substhermite}
\end{equation}
The explicit form of the Clifford-Hermite polynomials (see \cite{DBS3}, theorem 5) is given by
\[
\widetilde{CH}_{2t,M,k}(y) = \sum_{i=0}^t 2^{2t-2i}\left( \begin{array}{l}t\\i \end{array} \right)\frac{\Gamma(t+k+M/2)}{\Gamma(i+k+M/2)} y^{2i} 
\]
or, in the case where $M=-2n$, by
\[
\widetilde{CH}_{2t,-2n,k}(y) = \sum_{i=0}^t (-1)^{t-i} 2^{2t-2i}\left( \begin{array}{l}t\\i \end{array} \right)\frac{(n-k-i)!}{(n-k-t)!} y^{2i}.
\]

Plugging these expressions into the left-hand side of (\ref{substhermite}) yields
\begin{eqnarray*}
&& \sum_{i=0}^k \binom{k}{i} \frac{(n-j-i)!}{\Gamma(\frac{m}{2} + l-k-j-i)} \widetilde{CH}_{ 2k-2i,m,l-2k-j}(\uy) \widetilde{CH}_{ 2i,-2n,j}(\uyb)\\
&=& \sum_{i=0}^k  \sum_{p=0}^{k-i}  \sum_{q=0}^i 2^{2k-2p-2q} (-1)^{i-q} \binom{k}{i} \binom{k-i}{p} \binom{i}{q}   \frac{(n-j-q)!}{\Gamma(\frac{m}{2} + p +l - 2k-j)} \uy^{2p} \uyb^{2q}\\
&=& \sum_{p=0}^k  \sum_{q=0}^{k-p}  \sum_{i=q}^{k-p} 2^{2k-2p-2q} (-1)^{i-q} \binom{k}{i} \binom{k-i}{p} \binom{i}{q}   \frac{(n-j-q)!}{\Gamma(\frac{m}{2} + p +l -2k-j)} \uy^{2p} \uyb^{2q}\\
&=& \sum_{p=0}^k  \sum_{q=0}^{k-p} 2^{2k-2p-2q} (-1)^{q}  \frac{(n-j-q)!}{\Gamma(\frac{m}{2} + p +l-2k-j)} \uy^{2p} \uyb^{2q} \sum_{i=q}^{k-p} (-1)^{i} \binom{k}{i} \binom{k-i}{p} \binom{i}{q}  \\
&=&\sum_{p=0}^k  \binom{k}{p}  \frac{(n-j-k+p)!}{\Gamma(\frac{m}{2} + p +l-2k-j)} \uy^{2p} \uyb^{2k-2p} \\
&=&\sum_{i=0}^k \binom{k}{i} \frac{(n-j-i)!}{\Gamma(\frac{m}{2} + l-k-j-i)} \uy^{2k-2i} \uyb^{2i}\\
&=&f_{k,l-2k-j,j}(\uy^2,\uyb^2)
\end{eqnarray*}
since
\[
\sum_{i=q}^{k-p} (-1)^{i} \binom{k}{i} \binom{k-i}{p} \binom{i}{q} =\left\{ \begin{array}{ll} (-1)^{q} \binom{k}{p} &\mbox{if $k=p+q$}\\0& \mbox{otherwise.} \end{array} \right.
\]
This completes the proof of the theorem.
\end{proof}

As a consequence we obtain
\begin{corollary}
One has that
\[
H_l(\px) \exp(x^2/2) =  H_l(x) \exp(x^2/2)
\]
where $H_l(\px)$ is the operator obtained by substituting $x_i \rightarrow - \partial_{x_i}$, ${x \grave{}}_{2i} \rightarrow 2\partial_{{x \grave{}}_{2i-1}}$ and ${x \grave{}}_{2i-1} \rightarrow -2\partial_{{x \grave{}}_{2i}}$ in $H_l(x)$.
\label{CorDiffHarm}
\end{corollary}

Under mild assumptions on the super-dimension $M$ we can prove the converse of the previous corollary.
\begin{lemma}
Suppose $M \not \in -2 \mN$. If $p(x)$ is a homogeneous polynomial of degree $k$ satisfying
\[
p(\px) \exp(x^2/2) =  p(x) \exp(x^2/2)
\]
then $p(x) \in \cH_k$.
\end{lemma}

\begin{proof}
As $M \not \in -2 \mN$, $p$ can be developed in a Fischer decomposition (see \cite{DBS5}):
\[
p = \sum_{j=0}^{\lfloor k/2 \rfloor} x^{2j} H_{k-2j}
\]
with $H_{k-2j} \in \cH_{k-2j}$. We then have that
\begin{eqnarray*}
p(\px) \exp(x^2/2) &=&  \sum_{j=0}^{\lfloor k/2 \rfloor} \Delta^{j} H_{k-2j}(\px) \exp(x^2/2)\\
&=& \sum_{j=0}^{\lfloor k/2 \rfloor} \Delta^{j} H_{k-2j}(x) \exp(x^2/2)\\
&=& \sum_{j=0}^{\lfloor k/2 \rfloor}  \widetilde{CH}_{2j,M,k-2j}(x)  H_{k-2j}(x) \exp(x^2/2)
\end{eqnarray*}
where we have used corollary \ref{CorDiffHarm}.
From this expression it is clear that 
\[
p(\px) \exp(x^2/2) =  p(x) \exp(x^2/2)
\]
if and only if $p(x) = H_k(x)$.
\end{proof}

Now we can construct an eigenfunction basis of the super Fourier transform. We consider, as in section \ref{preliminaries}, the set of functions $\phi_{j,k,l}(x)$ defined by
\begin{eqnarray*}
\phi_{j,k,l}(x) = (\px + x )^j M_k^{(l)} \exp{x^2/2}= CH_{j,M,k}(x) M_k^{(l)} \exp{x^2/2}
\end{eqnarray*}
which form a basis for $\cS(\mR^m)_{m|2n}$ if $M \not \in -2\mN$.

We calculate the Fourier transform of these functions; we obtain
\begin{eqnarray*}
\cF_{m | 2n}^{\pm} (\phi_{j,k,l}(x) ) &=& \cF_{m | 2n}^{\pm} ((\px + x )^j M_k^{(l)} \exp{x^2/2})\\
&=& (\pm i )^j (\py + y )^j \cF_{m | 2n}^{\pm}  ( M_k^{(l)} \exp{x^2/2})\\
&=& (\pm i)^j (\pm i )^k (\py + y )^j M_k^{(l)} \exp{y^2/2}\\
&=&  (\pm i )^{j+k} \phi_{j,k,l}(y)\\
&=& e^{ \pm i(j+k)\frac{\pi}{2}} \phi_{j,k,l}(y)
\end{eqnarray*}
where we have used theorem \ref{fouriertransformsphericalharm} and lemma \ref{superfouriercalculus}.
This means that the Fourier transform rotates the basic functions over a multiple of a right angle.

Similarly, we obtain that the Fourier transform acts on the scalar basis $\psi_{j,k,l}(x)$ as
\[
\cF_{m | 2n}^{\pm} (\psi_{j,k,l}(x) ) = e^{ \pm i(2j+k)\frac{\pi}{2}} \psi_{j,k,l}(y).
\]

On the other hand we have that the functions $\phi_{j,k,l}(x)$ are solutions of a super harmonic oscillator (see \cite{DBS3}). This means that they satisfy
\[
\frac{1}{2} (\Delta - x^2) \phi_{j,k,l}(x) = \left( \frac{M}{2} + (j+k) \right)\phi_{j,k,l}(x).
\]
This implies that the operator exponential $\cK$ defined by
\[
\cK^{\pm}= \exp{ \pm \frac{i \pi}{4}(\Delta - x^2 - M)}
\]
satisfies
\[
\cK^{\pm}\phi_{j,k,l}(x) = e^{ \pm i(j+k)\frac{\pi}{2}} \phi_{j,k,l}(x)
\]
and thus equals the Fourier transform.
In this way we have proven the following theorem.

\begin{theorem}
The super Fourier transform is the operator exponential
\[
\cF^{\pm}_{m|2n}= \exp{ \pm \frac{i \pi}{4}(\Delta - x^2 - M)}.
\]
\end{theorem}

\section{Fundamental solution of the super Laplace operator}
\label{FundSolLaplaceOperator}

We do not aim at developing a complete theory of the Fourier transform on distributions in superspace in this section. We restrict ourselves to what is necessary for establishing the fundamental solution of the super Laplace operator $\Delta$.

We first calculate the Fourier transform of the super Dirac distribution. This yields
\begin{eqnarray*}
\cF^{\pm}_{m|2n} (\delta(x)) &=&  \cF^{\pm}_{m|0}(\delta(\ux)) \cF^{\pm}_{0|2n}(\frac{\pi^{n}}{n!}  \uxb^{2n})\\
&=& (2 \pi)^{-M/2}.
\end{eqnarray*}
Fourier transformation of the super Poisson equation $\Delta f(x) = \delta(x)$ then yields the algebraic equation
\[
-y^2 F(y) = (2 \pi)^{-M/2}
\]
with $F(y) = \cF^{\pm}_{m|2n}(f)$. Assuming $m >0$, we have that
\begin{eqnarray*}
(y^2)^{-1} &=& \uy^{-2} \left( 1 + \frac{\uyb^2}{\uy^2}\right)^{-1}\\
&=&\uy^{-2} \sum_{k =0}^{n} (-1)^k \left(\frac{\uyb^2}{\uy^2}\right)^k.
\end{eqnarray*}

It follows that the fundamental solution of the Laplace operator is given by the inverse Fourier transform of the following distribution
\[
F(y) = -(2 \pi)^{-M/2}\uy^{-2} \sum_{k =0}^{n} (-1)^k \left(\frac{\uyb^2}{\uy^2}\right)^k.
\]
We obtain that
\begin{eqnarray*}
f(x) &=& -(2 \pi)^{-M/2}\sum_{k =0}^{n} (-1)^k \cF^{\mp}_{m|0}(\uy^{-2k-2}) \cF^{\mp}_{0|2n}(\uyb^{2k})\\
&=&-(2 \pi)^{-M/2}\sum_{k =0}^{n} (-1)^k \cF^{\mp}_{m|0}(\uy^{-2k-2})  2^{2k-n} \frac{k!}{(n-k)!}  \uxb^{2n-2k}\\
&=& \pi^{n} \sum_{k =0}^{n}  \frac{2^{2k} k!}{(n-k)!} \cF^{\mp}_{m|0}\left(\frac{(2 \pi)^{-m/2} }{r^{2k+2}}\right)    \uxb^{2n-2k}
\end{eqnarray*}
with $r^2 = -\uy^2$. In this expression, $(2 \pi)^{-m/2} r^{-2k-2}$ is nothing else but the Fourier transform in distributional sense of the fundamental solution of the $(k+1)$-th power of the classical Laplace operator (see e.g. \cite{MR0435831}). Denoting this fundamental solution by $ \nu_{2k+2}^{m|0}$ we find that
\[
f(x) = \pi^{n} \sum_{k =0}^{n}  \frac{2^{2k} k!}{(n-k)!} \nu_{2k+2}^{m|0} \uxb^{2n-2k}.
\]
This result is the same as the one obtained in \cite{DBS6} using a completely different method. The Fourier method described in this section can also be used to construct a fundamental solution for the operator $\Delta^l$. Its main advantage over the one given in \cite{DBS6} is that no technical lemma of a combinatorial nature is needed.

\section{The fractional Fourier transform}
\label{FracFourierTrafo}

An extension of the fractional Fourier transform (see \cite{OZA} and references therein) to superspace may be introduced by
\[
\cF^{a}_{m|2n}= \exp{  \frac{i a \pi}{4}(\Delta - x^2 - M)},
\]
where $a \in [-1,1]$. The case $a = \pm 1$ is the Fourier transform studied in section \ref{gentransform}.

It is easy to see that this transform acts on the basis $\phi_{j,k,l}(x)$ as follows:
\[
\cF^{a}_{m|2n} (\phi_{j,k,l}(x)) = e^{  i a (j+k)\frac{\pi}{2}} \phi_{j,k,l}(x).
\]
The fractional Fourier transform thus rotates the basic functions over a multiple of the angle $\alpha = a \pi /2$.

The following theorem is easily proven.
\begin{theorem}
The fractional Fourier transform satisfies
\[
\cF^{a+b}_{m|2n} = \cF^{a}_{m|2n} \circ \cF^{b}_{m|2n} =  \cF^{b}_{m|2n} \circ \cF^{a}_{m|2n}.
\]
\end{theorem}

The inverse of the fractional Fourier transform is then immediately obtained:
\[
\cF^{a}_{m|2n} \circ \cF^{-a}_{m|2n} = \mbox{id}_{\cS(\mR^m)_{m|2n}} =   \cF^{-a}_{m|2n} \circ \cF^{a}_{m|2n}.
\]

Now we construct an integral representation of the fractional Fourier transform. In the following lemma we study the fermionic case with only two variables.
\begin{lemma}
The fractional Fourier transform in two anti-commuting variables has the following integral representation:
\begin{equation}
\cF^{a}_{0|2} =  \pi (1- e^{2i \alpha})\int_{B,x} \exp{\frac{ 2 e^{i \alpha} ({y \grave{}}_{2}{x \grave{}}_{1} - {y \grave{}}_{1} {x \grave{}}_{2}) + (1+ e^{2i \alpha})({x \grave{}}_{1}{x \grave{}}_{2} + {y \grave{}}_{1}{y \grave{}}_{2})}{2- 2e^{2i \alpha}}}
\label{oddfracftrafo}
\end{equation}
with $\alpha = a \pi/2$.
\label{lemmafermfracfourier}
\end{lemma}

\begin{proof}
By definition we have that
\[
\cF^{a}_{0|2} = \exp{ i \alpha H}
\]
with $\alpha = a \pi /2$ and with $H = 2 \partial_{{x\grave{}}_{1}} \partial_{{x\grave{}}_{2}} -{x \grave{}}_{1}{x \grave{}}_{2}/2 +1$. This operator $\cF^{a}_{0|2}$ acts on $\Lambda_2 = \mbox{span} (1 ,{x \grave{}}_{1} ,{x \grave{}}_{2} , {x \grave{}}_{1}{x \grave{}}_{2})$. Calculating the iterated action of $H$ on this basis yields
\[
\begin{array}{lll}
H^k (1) &=& \frac{1}{4} 2^k (2- {x \grave{}}_{1}{x \grave{}}_{2})\\
H^k ({x \grave{}}_{1}) &=& {x \grave{}}_{1}\\
H^k ({x \grave{}}_{2}) &=& {x \grave{}}_{2}\\
H^k ({x \grave{}}_{1}{x \grave{}}_{2}) &=& -\frac{1}{2} 2^k (2- {x \grave{}}_{1}{x \grave{}}_{2}).
\end{array}
\]
Using these results we find that the fractional Fourier transform acts on this basis as
\[
\begin{array}{lll}
\cF^{a}_{0|2} (1) &=& \frac{1}{2} (1+ e^{2i \alpha}) +  \frac{1}{4} (1- e^{2i \alpha}) {x \grave{}}_{1}{x \grave{}}_{2}\\
\cF^{a}_{0|2} ({x \grave{}}_{1}) &=& e^{i \alpha}{x \grave{}}_{1}\\
\cF^{a}_{0|2} ({x \grave{}}_{2}) &=& e^{i \alpha}{x \grave{}}_{2}\\
\cF^{a}_{0|2}({x \grave{}}_{1}{x \grave{}}_{2}) &=& (1- e^{2i \alpha}) +  \frac{1}{2} (1+ e^{2i \alpha}) {x \grave{}}_{1}{x \grave{}}_{2}.
\end{array}
\]
Requiring the integral operator
\[
\int_{B,x} F(x,y)
\]
with
\[
F(x,y) = \pi \left( f_0(y) + f_1(y) {x \grave{}}_{1} + f_2(y) {x \grave{}}_{2} + f_{12}(y) {x \grave{}}_{1} {x \grave{}}_{2} \right)
\]
and where $f_0(y),f_1(y),f_2(y)$ and $f_{12}(y)$ are still to be determined, to equal $\cF^{a}_{0|2}$ when acting on $\Lambda_2$, we obtain that
\[
F(x,y) = \pi \left( (1- e^{2i \alpha}) + \frac{1}{2}(1+ e^{2i \alpha}){y \grave{}}_{1} {y \grave{}}_{2}  + e^{i \alpha}[{y \grave{}}_{2} {x \grave{}}_{1} -{y \grave{}}_{1} {x \grave{}}_{2}] + [ \frac{1}{2}(1+ e^{2i \alpha}) + \frac{1}{4}(1- e^{2i \alpha}){y \grave{}}_{1} {y \grave{}}_{2} ]{x \grave{}}_{1} {x \grave{}}_{2} \right).
\]
It is now easy to check that expanding the kernel given in formula (\ref{oddfracftrafo}) yields the same result.
\end{proof}

Using this lemma we obtain the following integral representation of the general fractional Fourier transform.
\begin{theorem}
One has that
\[
\cF^{a}_{m|2n}= \left(\pi (1- e^{2i \alpha})\right)^{-M/2} \int_{\mR^m}  dV(\ux) \int_{B,x}  \exp{\frac{ - 4 e^{i \alpha} \langle x , y \rangle + (1+ e^{2i \alpha})(x^2 + y^2)}{2- 2e^{2i \alpha}}}. 
\]
\end{theorem}

\begin{proof}
By definition we have that
\[
\cF^{a}_{m|2n}= \exp{  i \alpha/2(\Delta - x^2 - M)}.
\]
We can evaluate this exponential as
\begin{eqnarray*}
\cF^{a}_{m|2n} &=& \exp{ i \alpha/2(\Delta_b  - \ux^2 - m)} \exp{ i \alpha/2(\Delta_f  - \uxb^2 +2n)}\\
&=& \cF^{a}_{m|0}  \exp{ i \alpha(4 \sum_j \partial_{{x\grave{}}_{2j-1}} \partial_{{x\grave{}}_{2j}}  - \frac{1}{2} \sum_j {x \grave{}}_{2j-1} {x \grave{}}_{2j} +2n)}\\
&=&\cF^{a}_{m|0}  \prod_j \exp{ i \alpha(2 \partial_{{x\grave{}}_{2j-1}} \partial_{{x\grave{}}_{2j}}  - \frac{1}{2} {x \grave{}}_{2j-1} {x \grave{}}_{2j} +1)}.
\end{eqnarray*}
The integral representation of the bosonic fractional Fourier transform is given by (see \cite{OZA})
\[
\cF^{a}_{m|0} =  \left(\pi (1- e^{2i \alpha})\right)^{-m/2} \int_{\mR^m}  dV(\ux)  \exp{\frac{ -4 e^{i \alpha} \langle \ux , \uy \rangle + (1+ e^{2i \alpha})(\ux^2 + \uy^2)}{2- 2e^{2i \alpha}}}.
\]
Combining this formula with lemma \ref{lemmafermfracfourier} gives the desired result.
\end{proof}

\begin{remark}
Note again that the kernel of the fractional Fourier transform is symmetric: $K(x,y) = K(y,x)$. Moreover, it would be impossible to introduce a fractional power of the Fourier transform in fermionic space using an orthogonal kernel. Indeed, if so we would not dispose of a Laplace operator nor of a generalization of the norm squared of a vector since ${x \grave{}}_i^2 = 0 = \partial_{{x\grave{}}_{i}}^2$.
\end{remark}

\begin{remark}
Note that we could also study a slightly more general transform, where we choose $\alpha$ independently for each commuting variable and each pair of anti-commuting variables.
\end{remark}

In the following lemma we give the basic calculus properties of the fractional Fourier transform.
\begin{lemma}
If $g \in \cS(\mR^m)_{m|2n}$, then the following relations hold:
\[
\begin{array}{lll}
\cF_{m| 2n}^{a} (\partial_{x_i} g) & = &[ \cos{\alpha} \partial_{y_i} - i \sin{\alpha} y_i ] \cF_{m | 2n}^{a} ( g) \\
\vspace{-3mm}\\
\cF_{m| 2n}^{a} (\partial_{{x\grave{}}_{2i}} g) & = & [ \cos{\alpha} \partial_{{y\grave{}}_{2i}} -  \frac{i}{2} \sin{\alpha} {y \grave{}}_{2i-1} ] \cF_{m | 2n}^{a} ( g)\\
\vspace{-3mm}\\
\cF_{m | 2n}^{a} (\partial_{{x\grave{}}_{2i-1}} g) & = & [ \cos{\alpha} \partial_{{y\grave{}}_{2i-1}} +  \frac{i}{2} \sin{\alpha} {y \grave{}}_{2i} ] \cF_{m | 2n}^{a} ( g)\\
\vspace{-3mm}\\
\cF_{m| 2n}^{a} (x_i g) & = & [- i \sin{\alpha} \partial_{y_i} + \cos{\alpha} y_i ] \cF_{m | 2n}^{a} ( g)\\
\vspace{-3mm}\\
\cF_{m | 2n}^{a} ({x\grave{}}_{2i} g) & = & [ 2i\sin{\alpha} \partial_{{y\grave{}}_{2i-1}} +  \cos{\alpha} {y \grave{}}_{2i} ] \cF_{m | 2n}^{a} ( g)\\
\vspace{-3mm}\\
\cF_{m| 2n}^{a} ({x\grave{}}_{2i-1} g) & = &[- 2i\sin{\alpha} \partial_{{y\grave{}}_{2i}} +  \cos{\alpha} {y \grave{}}_{2i-1} ] \cF_{m | 2n}^{a} ( g).
\end{array}
\]
As a consequence one has that
\[
\cF_{m| 2n}^{a} \left( (\px + x) g \right) = e^{i \alpha} (\px + x) \cF_{m| 2n}^{a} \left(  g \right).
\]
\label{fracfouriercalculus}
\end{lemma}

\begin{proof}
The formulae for $x_i$ and $\partial_{x_i}$ can be found in \cite{OZA}.

The formulae for ${x\grave{}}_{i}$ and $\partial_{{x\grave{}}_{i}}$ follow by explicitly computing the action of $\cF_{0| 2}^{a}$ on ${x\grave{}}_{i} g$ and on $\partial_{{x\grave{}}_{i}} g$ with $g \in \Lambda_2$.

Note that in the case where $\alpha = \pm \pi/2$ we reobtain lemma \ref{superfouriercalculus}.
\end{proof}

\section{The Radon transform in superspace}
\label{RadonTransform}

The classical Radon transform in $\mR^m$ is defined by
\[
\cR_{m|0} (f)(\underline{\omega}, p) = \int_{\mR^m}\delta(\ux \cdot \underline{\omega} - p) f(\ux) dV(\ux).
\]
It is a transform which maps a function defined on $\mR^m$ to the integral of this function over all hyperplanes in $\mR^m$, i.e. to a function defined on the half cilinder $\mS^{m-1}\times \mR^+$ with co-ordinates $( \underline{\omega},p)$. For the mathematical theory of the Radon transform we refer the reader to e.g. \cite{MR1723736}. A nice overview of its properties and applications can be found in \cite{MR709591}.

In this section we will use the connection between the Radon transform and the Fourier transform for its definition in superspace. This connection is expressed by the following theorem (see \cite{MR1723736}).
\begin{theorem}[central-slice]
One has that
\[
\cR_{m|0} (f)(\underline{\omega}, p) = (2 \pi)^{\frac{m}{2} -1} \int_{-\infty}^{\infty} e^{ipr} \cF_{m| 0}^{-}(f)(r \underline{\omega}) dr.
\]
\end{theorem}
This theorem says that the Radon transform of a function $f$ is obtained by first taking the Fourier transform of the function, followed by a second Fourier transform with respect to the radius $r$. As we have developed in the foregoing sections a general theory of Fourier transforms in superspace, we are able to use this relation as a definition of a super Radon transform. 

\begin{definition}
The super Radon transform of a function $f \in \cS(\mR^m)_{m|2n}$ is defined by
\[
\cR_{m|2n} (f) ( \omega, p) = (2 \pi)^{\frac{M}{2} -1} \int_{-\infty}^{\infty} e^{ipr} \left[\cF_{m| 2n}^{-}(f)(r \omega) \mod \omega^2 +1 \right]dr.
\]
\end{definition}
In this definition $\omega^2 =  \sum_{j=1}^n {\omega\grave{}}_{2j-1} {\omega\grave{}}_{2j}  -  \sum_{j=1}^m \omega_j^2 = -1$ is the algebraic relation defining the super-sphere (see \cite{DBS5}).
Also note that the transform is only well-defined if $m \neq 0$. There is no purely fermionic Radon transform.

This Radon transform has the following basic properties.
\begin{lemma}
If $f \in \cS(\mR^m)_{m|2n}$, then one has
\[
\begin{array}{lll}
\cR_{m| 2n} (\partial_{x_i} g) & = & \omega_i \frac{\partial}{\partial p}\cR_{m | 2n} ( g) \\
\vspace{-3mm}\\
\cR_{m| 2n} (\partial_{{x\grave{}}_{2i}} g) & = & \frac{1}{2} {\omega \grave{}}_{2i-1} \frac{\partial}{\partial p} \cR_{m | 2n} ( g)\\
\vspace{-3mm}\\
\cR_{m | 2n}(\partial_{{x\grave{}}_{2i-1}} g) & = & - \frac{1}{2} {\omega \grave{}}_{2i} \frac{\partial}{\partial p}\cR_{m | 2n} ( g).
\end{array}
\]
\end{lemma}

\begin{proof}
Immediate, using lemma \ref{superfouriercalculus}.
\end{proof}

Now we compute the generalized Radon transform of a well-chosen basis $\widetilde{\psi_{j,k,l}}(x)$ of $\cS(\mR^m) \otimes \Lambda_{2n}$. Indeed, if we put
\begin{eqnarray*}
\widetilde{\psi_{j,k,l}}(x) &=& (\px)^{2j} H_k^{(l)} \exp{x^2/2}\\
&=& \widetilde{CH}_{2j,M,k}(x) H_k^{(l)} \exp{x^2/2},
\end{eqnarray*}
we can calculate that
\begin{eqnarray*}
\cR_{m|2n} (\widetilde{\psi_{j,k,l}}) &=& (- \omega)^{2j} (\frac{\partial}{\partial p})^{2j} \cR_{m|2n} (H_k^{(l)} \exp{x^2/2})\\
&=& (-1)^j (2 \pi)^{\frac{M}{2} -1} (\frac{\partial}{\partial p})^{2j}\int_{-\infty}^{\infty} e^{ipr} \cF_{m| 2n}^{-}(H_k^{(l)} \exp{x^2/2})(r \omega) dr\\
&=& (-1)^j (2 \pi)^{\frac{M}{2} -1} (-i)^k (\frac{\partial}{\partial p})^{2j}\int_{-\infty}^{\infty} e^{ipr} H_k^{(l)}(\omega) r^k \exp{-r^2/2} dr\\
&=& (-1)^j (2 \pi)^{\frac{M}{2} -1} (-i)^k  (\frac{\partial}{\partial p})^{2j + k} (-i)^k
\int_{-\infty}^{\infty} e^{ipr}  \exp{-r^2/2} dr H_k^{(l)}(\omega)\\
&=& (-1)^j (2 \pi)^{\frac{M-1}{2}}   (-1)^{2j+k} (\frac{\partial}{\partial p})^{2j + k}
 \exp{-p^2/2}H_k^{(l)}(\omega)\\
&=&(-1)^j (2 \pi)^{\frac{M-1}{2}} \widetilde{H}_{2j+k}(p)
 \exp{-p^2/2}H_k^{(l)}(\omega)
\end{eqnarray*}
with $\widetilde{H}_{2j+k}(p)$ the classical Hermite polynomial of degree $2j+k$.

Hence we may conclude that the Radon transform maps the basis $\widetilde{\psi_{j,k,l}}(x)$ of $\cS(\mR^m) \otimes \Lambda_{2n}$ into $\cS(\mR) \otimes \left(\oplus_{i=0}^{\infty} \cH_i \right)$.

\section{Conclusions}
In this paper we have studied the Fourier transform in superspace from a new point of view, namely that of Clifford analysis. The main difference between our approach and previous work on this subject is the fact that we use a kernel,  the fermionic part of which is invariant under symplectic changes of variables, instead of orthogonal ones.
We have proven several basic properties of this transform. We have also constructed an eigenfunction basis of the transform using the Clifford-Hermite polynomials. This led us to an operator-exponential expression for the Fourier transform, enabling us to introduce a fractional Fourier transform. A closed form of this operator was subsequently obtained.
Finally we have used the Fourier transform to define a Radon transform in superspace using the central-slice theorem. Again we have shown that this transform behaves nicely with respect to the Clifford-Hermite polynomials.

\end{document}